\documentclass[12pt, twoside, leqno]{article}

\usepackage{amsmath,amsthm}
\usepackage{amssymb}
\usepackage{tikz}
\usepackage{enumitem}

\usepackage{graphicx}

\usepackage[T1]{fontenc}

\newtheorem{theorem}{Theorem}[section]
\newtheorem{corollary}[theorem]{Corollary}
\newtheorem{lemma}[theorem]{Lemma}

\theoremstyle{definition}

\numberwithin{equation}{section}

\begin{document}


\baselineskip=17pt


\title{On the  subgroup of $B_4$ that contains the kernel of Burau representation }

\author{A. Beridze, L. Davitadze}

\date{}

\maketitle


\renewcommand{\thefootnote}{}



\renewcommand{\thefootnote}{\arabic{footnote}}
\setcounter{footnote}{0}


\begin{abstract}
 It is known that there are braids $\alpha$ and $\beta$  in the braid group $B_4$,  such that  the group $\langle \alpha, \beta \rangle$ is a fee subgroup \cite{7}, which contains the kernel $K$ of the Burau map $\rho_4 : B_4 \to G	L\left(3, \mathbb{Z}[t,t^{-1}]\right)$ \cite{6}, \cite{4}.  In this paper we will prove that  $K$ is subgroup of $G=\langle \tau, \Delta \rangle $, where $\tau $ and $\Delta $ are fourth and square roots of the generator $\theta$ of the center $Z$ of the group $B_4$.  Consequently, we will write  elements of $K$ in terms of $\tau^i,~~i=1,2,3$ and $\Delta$. Moreover, we will show that the quotient group $G/Z$ is isomorphic to the free product $Z_4 *Z_2$.  
\end{abstract}

\section{Introduction}
Let $ \rho_4 :B_4\to GL\left(3,\mathbb{Z}\left[t,t^{-1}\right]\right)$
be  the reduced Burau representation of the braid group $B_4$. 
Consider the matrices $A=\rho_4(\alpha)$ and $B=\rho_4(\beta)$, where $\alpha={\sigma }_1{\sigma }_2{\sigma }^{-1}_3{\sigma }_1{\sigma }_2^{-1}{\sigma }_1^{-1}$ and $\beta={\sigma }_3{\sigma }^{-1}_1$ ($\sigma_i, ~i=1,2,3$ are standard generators of $B_4$). It is known that the group $\langle \alpha, \beta\rangle$ generated by $\alpha$ and $\beta$ is a free group, which contains the kernel of the Burau map $\rho_4 : B_4 \to G
L\left(3, \mathbb{Z}[t,t^{-1}]\right)$ \cite{1}, \cite{4}, \cite{6}, \cite{7}. Note that in \cite[Theorem 3.19]{1} it is shown the the kernel of the Burau map $\rho_4$ is subgroup of the free group $\langle X, Y \rangle$ of rank 2, where $X=\sigma_3 \sigma_1^{-1}$ and $Y=\sigma_2 \sigma_3 \sigma_1^{-1}\sigma_2^{-1}.$ On the other hand, since $\beta=X$ and $\alpha =  Y^{-1} \cdot X$, we have  $\langle X,Y \rangle =\langle \alpha, \beta\rangle$.

In the paper \cite {12} it is shown that there exists an order four matrix $T$, which satisfies the following equality:
\begin{equation} \label{1}
	A =T B T^{-1}, ~~~ A^{-1}=T^{-1}BT, ~~~ B^{-1}=T^2 BT^2.
\end{equation}
Using of these relations, we have shown that the braid $\tau =\sigma_1 \sigma_2 \sigma_3 \in B_4$ has the properties
\begin{equation} \label{2}
	\alpha =\tau^{-1}  \beta  \tau, ~~~ \alpha^{-1}=\tau\beta \tau^{-1}, ~~~ \beta^{-1}=\tau^2 \beta \tau^{-2}.
\end{equation}
On the other hand , $\beta = \Delta^{-1} \tau^2$, where  $\Delta =\sigma_1 \sigma_2 \sigma_3 \sigma_1 \sigma_2 \sigma_1$.  Consequently, we will obtain that  an element of $K$ have the form
\begin{equation} \label{3}
	\omega = \theta^m \tau^2   \Delta\tau^{i_1}  \Delta \tau^{i_2} \dots \tau^{i_k}\Delta\tau^2, ~ m \in \mathbb{Z}, i_i\in \{1,2,3\}.
\end{equation}
where $\theta$ is the generator of the center $Z$ of the group $B_4$.  Since $\rho_4(\theta)=t^4 I$, if $\bar{T}=\rho_4(\tau)$ and $D=\rho_4(\delta)$, then the Burau representation for $n=4$ is  faithful if and only if the product of the matrices of the form
\begin{equation} \label{4}
	t^{4m} \bar{T}^2  D \bar{T}^{i_1}  D \bar{T}^{i_2} \dots \bar{T}^{i_k} D\bar{T}^2 , ~ m \in \mathbb{Z}, i_i\in \{1,2,3\}
\end{equation}
is not the identity matrix.

\section{Subgroup generated by $\tau$ and $\Delta$}

Let $\sigma_1$, $\sigma_2$ and $\sigma_3$ be the standard generators of the braid group $B_4$. Consider the corresponding Burau matrices:
\begin{equation} \label{5}
	\small
	\rho_4 \left({\sigma }_1\right)=\left[\begin{array}{ccc}
		-t & t & 0 \\ 
		0 & 1 & 0 \\ 
		0 & 0 & 1 \end{array}
	\right], 
	\rho_4 \left({\sigma }_2\right)=\left[\begin{array}{ccc}
		1 & 0& 0 \\ 
		1 & -t & t \\ 
		0 & 0 & 1 \end{array}
	\right], 
	\rho_4 \left({\sigma }_3\right)=\left[\begin{array}{ccc}
		1 & 0 & 0 \\ 
		0 & 1 & 0 \\ 
		0 & 1 & -t \end{array}
	\right].
\end{equation}
Note that there is a minor different from standard matrices considered in \cite{1} and following \cite{3} we use left multiplication of matrices. 

In the paper \cite[Lemma 2.1]{3} is shown that for  $\alpha={\sigma }_1{\sigma }_2{\sigma }^{-1}_3{\sigma }_1{\sigma }_2^{-1}{\sigma }_1^{-1}$ and $\beta={\sigma }_3{\sigma }^{-1}_1$ and corresponding matrices
\begin{equation} \label{6}
	A=\rho_4 \left(\alpha\right)=\left[ \begin{array}{ccc}
		0 & 0 & -t^{-1} \\ 
		0 & -t & -t^{-1}+t \\ 
		-1 & 0 & -t^{-1}+1 \end{array}
	\right],
\end{equation}
\begin{equation} \label{7}
	~~~~  B=\rho_4 \left(\beta\right)=\left[ \begin{array}{ccc}
		-t^{-1} & 1 & 0 \\ 
		0 & 1 & 0 \\ 
		0 & 1 & -t \end{array}
	\right],
\end{equation}
there is an order four matrix 
\begin{equation} \label{8}
	T=\left[ \begin{array}{ccc}
		-1 & 1 & 0\\ 
		-1& 0 & 1 \\ 
		-1 & 0 & 0 \end{array}
	\right]
\end{equation}
such that
\begin{equation} \label{9}
	A =T B T^{-1}, ~~~ A^{-1}=T^{-1}BT, ~~~ B^{-1}=T^2 BT^2.
\end{equation}
Note that,  since we use {\bf the left multiplication of matrices}, using  \eqref{9} we obtain the following: 
\begin{lemma}\label{l.2.1} The braid $\tau =\sigma_1 \sigma_2 \sigma_3$ (see Figure 1) is  the element of $B_4$, such that 
	\begin{equation} \label{10}
		\bar{T}=\rho_4(\tau)=\left[ \begin{array}{ccc}
			-t & t & 0\\ 
			-t& 0 & t \\ 
			-t & 0 & 0 \end{array}
		\right]
	\end{equation}
	and the following condition is fulfilled:
	\begin{equation} \label{11}
		\alpha =\tau^{-1} \beta \tau, ~~~ \alpha^{-1}=\tau \beta \tau^{-1}, ~~~ \beta^{-1}=\tau^2 \beta \tau^{-2}.
	\end{equation} 
\end{lemma}
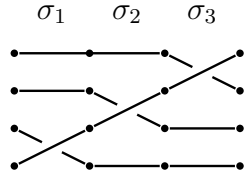
\begin{figure}[hbt!]
	\centering
	\begin{tikzpicture}[line width=.9pt, outer sep=0.5pt]
		\foreach \x in {0,1,2,3}
		\foreach \y in {0,1,2,3}
		\draw (1*\x,0.5*\y) node (\x\y)[circle,fill,inner sep=1pt]{};
		\draw  (02)--(12)--(21)--(31);
		\draw  (01)--(10)--(20)--(30);
		\draw  (03)--(23)--(32);
		\node at (0.5,0.25)[circle,fill=white,inner sep=3pt]{};
		\node at (1.5,0.75)[circle,fill=white,inner sep=3pt]{};
		\node at (2.5,1.25)[circle,fill=white,inner sep=3pt]{};
		\draw (00) -- (11)--(22)--(33);
		
		\node at (0.5,2.02){ {$\sigma_1$}};
		\node at (1.5,2.02){ {$\sigma_2$}};
		\node at (2.5,2.02){ {$\sigma_3$}};
		
	\end{tikzpicture}\\
	\caption{$\tau=\sigma_1\sigma_2\sigma_3$}
	\label{figure 1}
\end{figure}
\begin{proof} The equation \eqref{10} can be obtain by direct calculation. For  \eqref{11} we will use the following relations:
	\begin{equation} \label{12.1}
		\sigma^{-1}_{i+1}\sigma^{-1}_i\sigma_{i+1}=\sigma_{i}\sigma^{-1}_{i+1}\sigma^{-1}_{i},~~\sigma_{i+1}\sigma^{-1}_i\sigma^{-1}_{i+1}=\sigma^{-1}_{i}\sigma^{-1}_{i+1}\sigma_{i+1}
	\end{equation}
	\begin{equation} \label{13.1}
		\sigma^{-1}_{i+1}\sigma_i\sigma_{i+1}=\sigma_{i}\sigma_{i+1}\sigma^{-1}_{i},~~\sigma_{i+1}\sigma_i\sigma^{-1}_{i+1}=\sigma^{-1}_{i}\sigma_{i+1}\sigma_{i+1}.
	\end{equation}
	{\bf Case 1.}  $\alpha =\tau^{-1} \beta \tau$ (see Figure 2 and Figure 3):
	\[\tau^{-1}\beta\tau=(\sigma_3^{-1}\sigma_2^{-1}\sigma_1^{-1})(\sigma_3\sigma_1^{-1})(\sigma_1\sigma_2\sigma_3)=\sigma_3^{-1}\sigma_2^{-1}\sigma_1^{-1}\sigma_3\sigma_2\sigma_3=\]
	\[\sigma_3^{-1}\sigma_2^{-1}\sigma_1^{-1}\sigma_2\sigma_3\sigma_2=
	\sigma_3^{-1}\sigma_1\sigma_2^{-1}\sigma_1^{-1}\sigma_3\sigma_2=\sigma_1\sigma_3^{-1}\sigma_2^{-1}\sigma_3\sigma_1^{-1}\sigma_2=\]
	\begin{equation} \label{14.1}
		\sigma_1\sigma_2\sigma_3^{-1}\sigma_2^{-1}\sigma_1^{-1}\sigma_2=\sigma_1\sigma_2\sigma_3^{-1}\sigma_1\sigma_2^{-1}\sigma_1^{-1}=\alpha.
	\end{equation}
	
	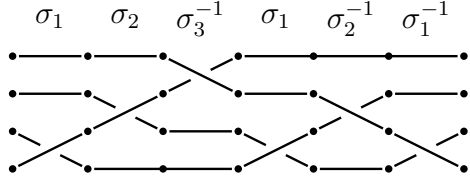
\begin{figure}[hbt!]
		\centering
		\begin{tikzpicture}[line width=.9pt, outer sep=1pt]
			\foreach \x in {0,1,2,3,4,5,6}
			\foreach \y in {0,1,2,3}
			\draw (1*\x,0.5*\y) node (\x\y)[circle,fill,inner sep=1pt]{};
			\draw (02) -- (12)--(21)--(31)--(40)--(50)--(61);
			\node at (1.5,0.75)[circle,fill=white,inner sep=3pt]{};
			\node at (3.5,0.25)[circle,fill=white,inner sep=3pt]{};
			\node at (5.5,0.25)[circle,fill=white,inner sep=3pt]{};
			\draw (01)--(10)--(30)--(41)--(52)--(62);
			\node at (0.5,0.25)[circle,fill=white,inner sep=3pt]{};
			\node at (4.5,0.75) [circle,fill=white,inner sep=3pt]{};
			\draw (00)--(11)--(22)--(33)--(63);
			\node at (2.5,1.25)[circle,fill=white,inner sep=3pt]{};

			\draw (03)--(13)--(23)--(32)--(42)--(51)--(60) ;
			
			\node at (0.5,2){ {$\sigma_1$}};
			\node at (1.5,2){ {$\sigma_2$}};   
			\node at (2.5,2){ {$\sigma_3^{-1}$}};
			\node at (3.5,2){{$\sigma_1$}};
			\node at (4.5,2){ {$\sigma_2^{-1}$}};
			\node at (5.5,2){ {$\sigma_1^{-1}$}};
		\end{tikzpicture}\\
		\caption{$\alpha={\sigma }_1{\sigma }_2{\sigma }^{-1}_3{\sigma }_1{\sigma }_2^{-1}{\sigma }_1^{-1}$}
		\label{figure 1}
	\end{figure}
	
	\begin{figure}[hbt!]
		\centering
		\begin{tikzpicture}[line width=.9pt, outer sep=1pt]
			\foreach \x in {0,1,2,3,4,5,6,7,8}
			\foreach \y in {0,1,2,3}
			\draw (1*\x,0.5*\y) node (\x\y)[circle,fill,inner sep=1pt]{};
			
			\draw [thick] (00)--(20)--(31)--(41)--(50)--(61)--(72)--(83);
			\draw [thick] (01)--(11)--(22)--(32)--(43)--(73)--(82);
			\draw [thick] (02)--(13)--(33)--(42)--(62)--(71)--(81);
			\draw [thick] (03)--(12)--(21)--(30)--(40)--(51)--(60)--(80);
			\node at (0.5,1.25)[circle,fill=white,inner sep=3pt]{};
			\node at (1.5,0.75)[circle,fill=white,inner sep=3pt]{};
			\node at (2.5,0.25)[circle,fill=white,inner sep=3pt]{};
			\node at (3.5,1.25)[circle,fill=white,inner sep=3pt]{};
			\node at (4.5,0.25)[circle,fill=white,inner sep=3pt]{};
			\node at (5.5,0.25)[circle,fill=white,inner sep=3pt]{};
			\node at (6.5,0.75)[circle,fill=white,inner sep=3pt]{};
			\node at (7.5,1.25)[circle,fill=white,inner sep=3pt]{};
			\draw [thick] (03)--(12)--(21)--(30) (32)--(43) (40)--(51) (50)--(61)--(72)--(83);
			
			\node at (4.5,0.25)[circle,fill=white,inner sep=3pt]{};
			\draw [thick]  (41)--(50);
			
			\node at (0.5,1.8){$\sigma_3^{-1}$};
			\node at (1.5,1.8){$\sigma_2^{-1}$};
			\node at (2.5,1.8){$\sigma_1^{-1}$};
			\node at (3.5,1.8){$\sigma_3$};
			\node at (4.5,1.8){$\sigma_1^{-1}$};
			\node at (5.5,1.8){$\sigma_1$};
			\node at (6.5,1.8){$\sigma_2$};
			\node at (7.5,1.8){$\sigma_3$};
			
		\end{tikzpicture}	
		\caption{$\tau^{-1}\beta \tau$}
		\label{figure 1}
	\end{figure}
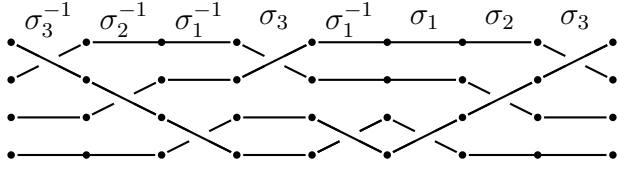
	\newpage
	{\bf Case 2.}  $\alpha^{-1} =\tau\beta \tau^{-1}$ (see Figure 4 and Figure 5):
	\[\tau\beta \tau^{-1}=(\sigma_1\sigma_2\sigma_3)(\sigma_3\sigma_1^{-1})(\sigma_3^{-1}\sigma_2^{-1}\sigma_1^{-1})=\]
	\begin{equation} \label{15.1}
		\sigma_1\sigma_2\sigma_3\sigma_1^{-1}\sigma_2^{-1}\sigma_1^{-1}=\sigma_1\sigma_2\sigma_1^{-1}\sigma_3\sigma_2^{-1}\sigma_1^{-1}=\alpha^{-1}.
	\end{equation}
	
	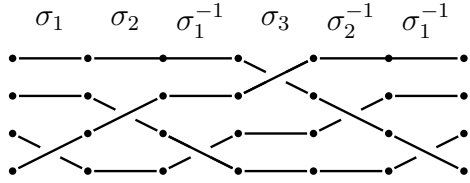
\begin{figure}[hbt!]
		\centering
		\begin{tikzpicture}[line width=.9pt, outer sep=1pt]
			\foreach \x in {0,1,2,3,4,5,6}
			\foreach \y in {0,1,2,3}
			\draw (1*\x,0.5*\y) node (\x\y)[circle,fill,inner sep=1pt]{};
			\draw (02) -- (12)--(21)--(30)--(40)--(50)--(61);
			\node at (1.5,0.75)[circle,fill=white,inner sep=3pt]{};
			\node at (3.5,0.25)[circle,fill=white,inner sep=3pt]{};
			\node at (5.5,0.25)[circle,fill=white,inner sep=3pt]{};
			\draw (01)--(10)--(20)--(31)--(41)--(52)--(62);
			\node at (0.5,0.25)[circle,fill=white,inner sep=3pt]{};
			\node at (4.5,0.75) [circle,fill=white,inner sep=3pt]{};
			\draw (00)--(11)--(22)--(32)--(43)--(63);
			\node at (2.5,1.25)[circle,fill=white,inner sep=3pt]{};

			\draw (03)--(13)--(33)--(42)--(51)--(60) ;
			\node at (2.5,0.25)[circle,fill=white,inner sep=3pt]{};
			\node at (3.5,1.25)[circle,fill=white,inner sep=3pt]{};
			\draw (21)--(30) (32)--(43);
			\node at (5.5,1.25)[circle,fill=white,inner sep=3pt]{};

			\node at (0.5,2){ {$\sigma_1$}};
			\node at (1.5,2){ {$\sigma_2$}};   
			\node at (2.5,2){ {$\sigma_1^{-1}$}};
			\node at (3.5,2){{$\sigma_3$}};
			\node at (4.5,2){ {$\sigma_2^{-1}$}};
			\node at (5.5,2){ {$\sigma_1^{-1}$}};
		\end{tikzpicture}\\
		\caption{$\alpha^{-1}={\sigma }_1{\sigma }_2{\sigma }^{-1}_1{\sigma }_3{\sigma }_2^{-1}{\sigma }_1^{-1}$}
		\label{figure 1}
	\end{figure}

	\begin{figure}[hbt!]
		\centering
		\begin{tikzpicture}[line width=.9pt, outer sep=1pt]
			\foreach \x in {0,1,2,3,4,5,6,7,8}
			\foreach \y in {0,1,2,3}
			\draw (1*\x,0.5*\y) node (\x\y)[circle,fill,inner sep=1pt]{};
			
			\draw [thick] (01)--(10)--(40)--(51)--(61)--(72)--(82);
			
			\node at (0.5,0.25)[circle,fill=white,inner sep=3pt]{};
			\node at (4.5,0.25)[circle,fill=white,inner sep=3pt]{};
			\node at (6.5,0.75)[circle,fill=white,inner sep=3pt]{};
			
			\draw [thick] (02)--(12)--(21)--(41)--(50)--(70)--(81);
			
			\node at (1.5,0.75)[circle,fill=white,inner sep=3pt]{};
			\node at (7.5,0.25)[circle,fill=white,inner sep=3pt]{};
			
			\draw [thick] (03)--(23)--(32)--(43)--(53)--(62)--(71)--(80);
			\node at (2.5,1.25)[circle,fill=white,inner sep=3pt]{};
			\node at (5.5,1.25)[circle,fill=white,inner sep=3pt]{};
			
			\draw [thick] (00)--(11)--(22)--(33)--(42)--(52)--(63)--(83); 
			
			\node at (3.5,1.25)[circle,fill=white,inner sep=3pt]{};
			
			\draw [thick]  (32)--(43);
			\node at (5.5,1.25)[circle,fill=white,inner sep=3pt]{};
			\draw [thick] (53)--(62);
			
			\node at (0.5,1.8){$\sigma_1$};
			\node at (1.5,1.8){$\sigma_2$};
			\node at (2.5,1.8){$\sigma_3$};
			\node at (3.5,1.8){$\sigma_3$};
			\node at (4.5,1.8){$\sigma_1^{-1}$};
			\node at (5.5,1.8){$\sigma_3^{-1}$};
			\node at (6.5,1.8){$\sigma_2^{-1}$};
			\node at (7.5,1.8){$\sigma_1^{-1}$};
		\end{tikzpicture}\\
		\caption{$\tau \beta \tau^{-1} $}
		\label{figure 1}
	\end{figure}
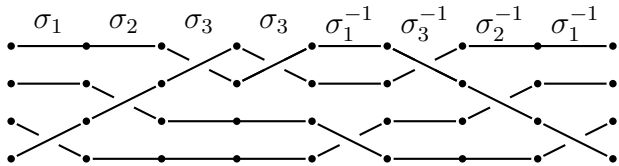
	
	{\bf Case 3.}  $\beta^{-1} =\tau^2\beta \tau^{-2}$ (see Figure 6 and Figure 7):
	\[\tau^2\beta \tau^{-2}=(\sigma_1\sigma_2\sigma_3\sigma_1\sigma_2\sigma_3)(\sigma_3\sigma_1^{-1})(\sigma_3^{-1}\sigma_2^{-1}\sigma_1^{-1}\sigma_3^{-1}\sigma_2^{-1}\sigma_1^{-1})=\]
	\[\sigma_1\sigma_2\sigma_3\sigma_1\sigma_2\sigma_3\sigma_1^{-1}\sigma_2^{-1}\sigma_1^{-1}\sigma_3^{-1}\sigma_2^{-1}\sigma_1^{-1}=\]
	\[\sigma_1\sigma_2\sigma_3\sigma_1\sigma_2\sigma_1^{-1}\sigma_3\sigma_2^{-1}\sigma_3^{-1}\sigma_1^{-1}\sigma_2^{-1}\sigma_1^{-1}=\]
	\[\sigma_1\sigma_2\sigma_3\sigma_2^{-1}\sigma_1\sigma_2\sigma_2^{-1}\sigma_3^{-1}\sigma_2\sigma_2^{-1}\sigma_1^{-1}\sigma_2^{-1}=\]
	\begin{equation} \label{15.1}
		\sigma_1\sigma_2\sigma_3\sigma_2^{-1}\sigma_3^{-1}\sigma_2^{-1}=
		\sigma_1\sigma_2\sigma_3\sigma_3^{-1}\sigma_2^{-1}\sigma_3^{-1}=
		\sigma_1\sigma_3^{-1}=\beta^{-1}.
	\end{equation}
	
	\begin{figure}[hbt!]
		\centering
		\begin{tikzpicture}[line width=.9pt, outer sep=1pt]
			\foreach \x in {0,1,2}
			\foreach \y in {0,1,2,3}
			\draw (1*\x,0.5*\y) node (\x\y)[circle,fill,inner sep=1pt]{};
			
			\draw [thick] (01)--(10)--(20);
			\draw [thick] (02)--(12)--(23);
			\node at (0.5,0.25) [circle,fill=white,inner sep=3pt]{};
			\node at (1.5,1.25) [circle,fill=white,inner sep=3pt]{};
			
			\draw [thick] (03)--(13)--(22);
			\draw [thick] (00)--(11)--(21);

			\node at (0.5,1.79) {$\sigma_1$};
			\node at (1.5,1.79) {$\sigma_3^{-1}$};   
		\end{tikzpicture}\\
		
		\caption{$\beta^{-1}=\sigma_1\sigma_3^{-1}$}
		\label{figure 1}
	\end{figure}
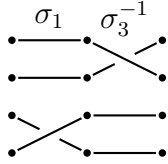

	\begin{figure}[hbt!]
		\centering	
		\begin{tikzpicture}[line width=.9pt, outer sep=1pt]
			\foreach \x in {0,1,2,3,4,5,6,7,8,9,10,11,12,13,14}
			\foreach \y in {0,1,2,3}
			\draw (1*\x,0.5*\y) node (\x\y)[circle,fill,inner sep=1pt]{};
			
			\draw [thick] (00)--(11)--(22)--(33)--(53)--(62)--(73)--(83)--(92)--(101)--(110)--(120);
			
			\node at (8.5,1.25) [circle,fill=white,inner sep=3pt]{};
			
			\draw [thick] (01)--(10)--(30)--(41)--(52)--(63)--(72)--(82)--(93)--(113)--(122)--(131)--(140);
			\draw [thick] (02)--(12)--(21)--(31)--(40)--(70)--(81)--(91);
			\draw [thick] (03)--(23)--(32)--(42)--(51)--(71)--(80)--(90);
			
			\node at (13.5,0.25) [circle,fill=white,inner sep=3pt]{};
			\node at (11.5,1.25) [circle,fill=white,inner sep=3pt]{};
			\node at (10.5,0.25) [circle,fill=white,inner sep=3pt]{};
			\node at (0.5,0.25) [circle,fill=white,inner sep=3pt]{};
			\node at (1.5,0.75) [circle,fill=white,inner sep=3pt]{};
			\node at (2.5,1.25) [circle,fill=white,inner sep=3pt]{};
			\node at (3.5,0.25) [circle,fill=white,inner sep=3pt]{};
			\node at (4.5,0.75) [circle,fill=white,inner sep=3pt]{};
			\node at (5.5,1.25) [circle,fill=white,inner sep=3pt]{};
			\node at (6.5,1.25) [circle,fill=white,inner sep=3pt]{};
			\node at (7.5,0.25) [circle,fill=white,inner sep=3pt]{};
			\node at (8.5,0.25) [circle,fill=white,inner sep=3pt]{};
			\node at (9.5,0.75) [circle,fill=white,inner sep=3pt]{};
			\node at (10.5,1.25) [circle,fill=white,inner sep=3pt]{};
			\node at (11.5,0.25) [circle,fill=white,inner sep=3pt]{};
			\node at (12.5,0.75) [circle,fill=white,inner sep=3pt]{};
			\node at (13.5,1.25) [circle,fill=white,inner sep=3pt]{};
			
			\draw [thick] (00)--(11)--(22)--(33) (30)--(41)--(52)--(63) (62)--(73) (71)--(80)   (91)--(102)--(112)--(123)--(143)  (90)--(100)--(111)--(121)--(132)--(142)  (120)--(130)--(141) ;

			\node at (13.5,0.25) [circle,fill=white,inner sep=3pt]{};
			\node at (12.5,0.75) [circle,fill=white,inner sep=3pt]{};
			\node at (11.5,1.25) [circle,fill=white,inner sep=3pt]{};
			\node at (10.5,0.25) [circle,fill=white,inner sep=3pt]{};
			\node at (9.5,0.75) [circle,fill=white,inner sep=3pt]{};
			\node at (8.5,1.25) [circle,fill=white,inner sep=3pt]{};
			
			\draw [thick] (83)--(92)--(101)--(110)  (113)--(122)--(131)--(140) ;

			\node at (0.5,1.8){$\sigma_1$};
			\node at (1.5,1.8){$\sigma_2$};
			\node at (2.5,1.8){$\sigma_3$};
			\node at (3.5,1.8){$\sigma_1$};
			\node at (4.5,1.8){$\sigma_2$};
			\node at (5.5,1.8){$\sigma_3$};
			\node at (6.5,1.8){$\sigma_3$};
			\node at (7.5,1.8){$\sigma_1^{-1}$};
			\node at (8.5,1.8){$\sigma_3^{-1}$};
			\node at (9.5,1.8){$\sigma_2^{-1}$};
			\node at (10.5,1.8){$\sigma_1^{-1}$};
			\node at (11.5,1.8){$\sigma_3^{-1}$};
			\node at (12.5,1.8){$\sigma_2^{-1}$};
			\node at (13.5,1.8){$\sigma_1^{-1}$};
		\end{tikzpicture}\\
		\caption{$\tau^2\beta \tau^{-2}$}
		\label{figure 1}
	\end{figure}
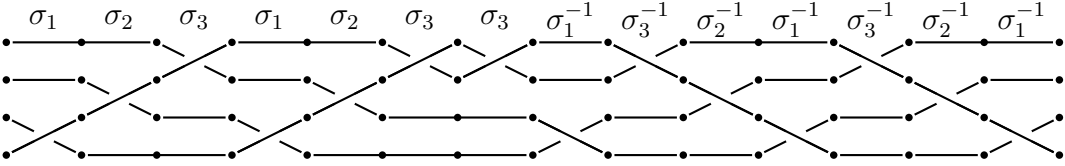
	
\end{proof}
\newpage
\begin{lemma}\label{l.2.2}  The braid $\Delta=\sigma_1 \sigma_2 \sigma_3\sigma_1\sigma_2\sigma_1$ (see Figure 8) is  the element of $B_4$, such that 
	\begin{equation} \label{12}
		\beta=\Delta^{-1}\tau^2, ~~\alpha=\tau^{-1}\Delta^{-1} \tau^3,~~\alpha^{-1}=\tau\Delta^{-1} \tau,~~\beta^{-1}=\tau^2 \Delta^{-1}.
	\end{equation} 
\end{lemma}
\begin{figure}[hbt!]
	\centering
	\begin{tikzpicture}[line width=.9pt, outer sep=0.5pt]
		\foreach \x in {0,1,2,3,4,5,6}
		\foreach \y in {0,1,2,3}
		\draw (1*\x,0.5*\y) node (\x\y)[circle,fill,inner sep=1pt]{};
		\draw  (01)--(10)--(20)--(30);
		
		\node at (4.5,0.75)[circle,fill=white,inner sep=3pt]{};
		\draw  (03)--(23)--(32)--(42)--(51)--(60);
		\draw  (02)--(12)--(21)--(31)--(40)--(50);
		
		\node at (0.5,0.25)[circle,fill=white,inner sep=3pt]{};
		\node at (1.5,0.75)[circle,fill=white,inner sep=3pt]{};
		\node at (2.5,1.25)[circle,fill=white,inner sep=3pt]{};
		\node at (5.5,0.25)[circle,fill=white,inner sep=3pt]{};
		\node at (4.5,0.75)[circle,fill=white,inner sep=3pt]{};
		\node at (3.5,0.25)[circle,fill=white,inner sep=3pt]{};
		\draw  (50)--(61);
		\draw  (50)--(61);
		\draw (30) -- (41)--(52)--(62);
		\draw  (00)--(11)--(22)--(33)--(43)--(53)--(63); 
		
		\node at (0.5,2.02){ {$\sigma_1$}};
		\node at (1.5,2.02){ {$\sigma_2$}};
		\node at (2.5,2.02){ {$\sigma_3$}};
		\node at (3.5,2.02){ {$\sigma_1$}};
		\node at (4.5,2.02){ {$\sigma_2$}};
		\node at (5.5,2.02){ {$\sigma_1$}};
		
	\end{tikzpicture}\\
	\caption{$\Delta=\sigma_1\sigma_2\sigma_3\sigma_1\sigma_2\sigma_1$}
	\label{figure 1}
\end{figure}
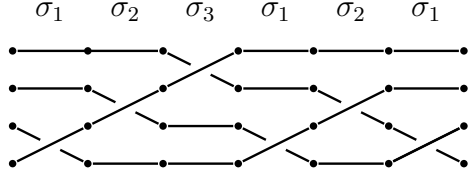
\begin{proof}  Let us show that $\beta=\Delta^{-1}\tau^2$ (see figure 9 and figure 10):
	\begin{equation} \label{13}
		\small
		\Delta^{-1}\tau^2=(\sigma_1^{-1}\sigma_2^{-1}\sigma_1^{-1} \sigma_3^{-1}\sigma_2^{-1}\sigma_1^{-1})(\sigma_1\sigma_2\sigma_3) (\sigma_1\sigma_2\sigma_3)=\sigma_1^{-1}\sigma_3=\sigma_3\sigma_1^{-1}=\beta
	\end{equation} 
	
	\begin{figure}[hbt!]
		\centering
		\begin{tikzpicture}[line width=.9pt, outer sep=1pt]
			\foreach \x in {0,1,2}
			\foreach \y in {0,1,2,3}
			\draw (1*\x,0.5*\y) node (\x\y)[circle,fill,inner sep=1pt]{};
			
			\draw [thick] (03)--(12)--(22);
			\draw [thick] (00)--(10)--(21);
			\node at (1.5,0.25) [circle,fill=white,inner sep=3pt]{};
			\node at (0.5,1.25) [circle,fill=white,inner sep=3pt]{};
			\draw [thick] (01)--(11)--(20);
			\draw [thick] (02)--(13)--(23);

			\node at (0.5,1.79) {$\sigma_3$};
			\node at (1.5,1.79) {$\sigma_1^{-1}$};   
		\end{tikzpicture}\\
		
		\caption{$\beta=\sigma_3\sigma_1^{-1}$}
		\label{figure 1}
	\end{figure}
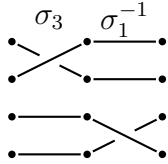

	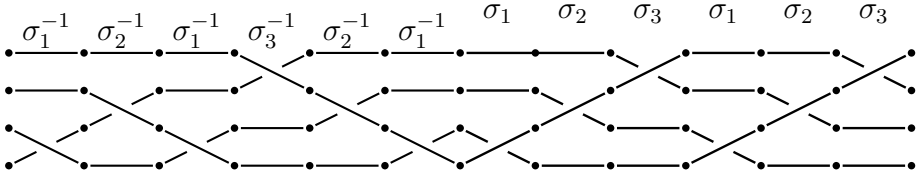
\begin{figure}[hbt!]
		\centering	
		\begin{tikzpicture}[line width=.9pt, outer sep=1pt]
			\foreach \x in {0,1,2,3,4,5,6,7,8,9,10,11,12}
			\foreach \y in {0,1,2,3}
			\draw (1*\x,0.5*\y) node (\x\y)[circle,fill,inner sep=1pt]{};
			
			\draw [thick] (00)--(11)--(22)--(32);
			\node at (0.5,0.25) [circle,fill=white,inner sep=3pt]{};
			\node at (1.5,0.75) [circle,fill=white,inner sep=3pt]{};
			\node at (2.5,1.25) [circle,fill=white,inner sep=3pt]{};
			\draw [thick] (01)--(10)--(20)--(31);
			\node at (2.5,0.25) [circle,fill=white,inner sep=3pt]{};
			\draw [thick] (02)--(12)--(21)--(30);
			\draw [thick] (03)--(13)--(23)--(33);
			
			\node at (0.5,1.8){$\sigma_1^{-1}$};
			\node at (1.5,1.8){$\sigma_2^{-1}$};
			\node at (2.5,1.8){$\sigma_1^{-1}$};
			
			\draw [thick] (30)--(40)--(50)--(61);

			\node at (5.5,1.25) [circle,fill=white,inner sep=3pt]{};
			\node at (5.5,0.25) [circle,fill=white,inner sep=3pt]{};
			\draw [thick] (31)--(41)--(52)--(62);
			\node at (4.5,0.75) [circle,fill=white,inner sep=3pt]{};
			\draw [thick] (32)--(43)--(53)--(63);
			\node at (3.5,1.25) [circle,fill=white,inner sep=3pt]{};
			
			\draw [thick] (33)--(42)--(51)--(60);
			
			\node at (3.5,1.8){$\sigma_3^{-1}$};
			\node at (4.5,1.8){$\sigma_2^{-1}$};
			\node at (5.5,1.8){$\sigma_1^{-1}$};

			\draw  (62)--(72)--(81)--(91);
			\draw  (61)--(70)--(80)--(90);
			\draw  (63)--(83)--(92);
			\node at (6.5,0.25)[circle,fill=white,inner sep=3pt]{};
			\node at (7.5,0.75)[circle,fill=white,inner sep=3pt]{};
			\node at (8.5,1.25)[circle,fill=white,inner sep=3pt]{};
			\draw (60) -- (71)--(82)--(93);
			
			\node at (6.5,2.02){ {$\sigma_1$}};
			\node at (7.5,2.02){ {$\sigma_2$}};
			\node at (8.5,2.02){ {$\sigma_3$}};
			
			\draw  (92)--(102)--(111)--(121);
			\draw  (91)--(100)--(110)--(120);
			\draw  (93)--(103)--(113)--(122);
			\node at (9.5,0.25)[circle,fill=white,inner sep=3pt]{};
			\node at (10.5,0.75)[circle,fill=white,inner sep=3pt]{};
			\node at (11.5,1.25)[circle,fill=white,inner sep=3pt]{};
			\draw (90) -- (101)--(112)--(123);
			
			\node at (9.5,2.02){ {$\sigma_1$}};
			\node at (10.5,2.02){ {$\sigma_2$}};
			\node at (11.5,2.02){ {$\sigma_3$}};
		\end{tikzpicture}\\
		\caption{$\Delta^{-1}\tau^2$}
		\label{figure 1}
	\end{figure}
	On the other hand, by Lemma	\ref{l.2.1} and the equality $\beta=\Delta^{-1}\tau^2$ implies other qualities  of \eqref{12}.
\end{proof}

\begin{theorem}\label{t.2.3}  The kernel $K$ of the Burau representation 
	\begin{equation} \label{14}
		\rho_4:B_4\to GL\left(3,\mathbb{Z}\left[t,t^{-1}\right]\right)
	\end{equation}
	is contained in the subgroup $G$ of $B_4$ generated by the elements $\tau$ and $\Delta$. Moreover, if the kernel $K$ is  not-trivial, then there exists a non-identity element $\omega \in K$ that can be written in the form	
	\begin{equation} \label{15}
		\omega = \theta^m \tau^2   \Delta\tau^{i_1}  \Delta \tau^{i_2} \dots \tau^{i_k}\Delta \tau^2, ~ m \in \mathbb{Z}, i_i\in \{1,2,3\}.
	\end{equation}
\end{theorem}
\begin{proof} By the corollary 3.2 \cite{4} any kernel element can  be written as word in the Bokut–Vesnin (Gorin-Lin)  generators $\alpha,$ $\beta,$ $\alpha^{-1}$ and $\beta^{-1}$. Let  $\omega$ be a non-identity kernel element, written as an irreducible non-empty word in letters $\alpha,$ $\beta,$ $\alpha^{-1}$ and $\beta^{-1}$. We can assume that $\omega$ has the suffix $\beta$ and prefix $\beta^{-1}$.  If not we can conjugate it by some power of  $\beta$. In this case, by substitution $\beta=\Delta^{-1}\tau^2,$ $\alpha=\tau^{-1}\Delta^{-1} \tau^3,$ $\alpha^{-1}=\tau\Delta^{-1} \tau,$ $\beta^{-1}=\tau^2 \Delta^{-1}$  and using the property that $\tau^{4m}=\Delta^{2m}=\theta^m$ commutes all elements of $B_4$, we can reduce $\omega$ in the form \eqref{15}.
\end{proof}
\begin{corollary}\label{c.2.4}  The Burau representation for $n=4$ is  faithful if and only if the product of the matrices of the form
	\begin{equation} \label{16}
		t^{4m} \bar{T}^2  D \bar{T}^{i_1}  D \bar{T}^{i_2} \dots \bar{T}^{i_k} D\bar{T}^2, ~ m \in \mathbb{Z}, i_i\in \{1,2,3\}
	\end{equation}
	is not the identity matrix, where  $\bar{T}=\rho_4(\tau)$ and $D=\rho_4(\Delta)$. 
\end{corollary}
\begin{corollary}\label{c.2.5}  Let $G$ be the subgroup of $B_4$ generated by $\tau$ and $\Delta$ and $Z$ be the center, then a representation of $G/Z$ is given by:
	\begin{equation} \label{17}
		\langle \tau, \Delta | \tau^4=\Delta^2=1 \rangle,
	\end{equation}
	where $\tau=\sigma_1\sigma_2\sigma_3$ and $\Delta=\sigma_1\sigma_2\sigma_3\sigma_1\sigma_2\sigma_1$.
\end{corollary}


\end{document}